\documentclass[10pt]{article}
\usepackage[cp1251]{inputenc}
\usepackage{amssymb}
\usepackage{amsfonts}
\usepackage{amsthm}

\usepackage{array}
\makeatletter
\@addtoreset{equation}{section}
\makeatother

\begin{document}

\theoremstyle{plain}
\newtheorem{Th}{Theorem}[section]
\theoremstyle{definition}
\newtheorem{lemma}[Th]{Lemma}
\newtheorem{Def}[Th]{Definition}
\newtheorem{Cor}[Th]{Corollary}
\theoremstyle{remark}
\newtheorem{remark}[Th]{Remark}
\theoremstyle{proof}

\renewcommand{\thesection}{\arabic{section}.}
\renewcommand{\theTh}{\thesection\arabic{Th}}
\renewcommand{\theequation}{\thesection\arabic{equation}}
\newcommand{\Om}{\Omega}
\newcommand{\Omm}{\tilde\Omega}
\newcommand{\Ga}{\Gamma}
\newcommand{\Go}{\Gamma_0}
\newcommand{\de}{\delta}
\newcommand{\phii}{\tilde\phi}
\newcommand{\om}{\omega}
\newcommand{\pa}{\partial}

\sloppy
\binoppenalty=10000
\relpenalty=10000

\title{\bf On attractors of $m$-Hessian evolutions}
\date{}
\maketitle
\author{Nina Ivochkina}

\centerline{St.-Petersburg State University}
 \centerline{Universitetskaya nab., 7-9}
 \centerline{199034 St.Petersburg, Russia}
\vskip .1in
\centerline{ninaiv@NI1570.spb.edu}

\vskip .2in
\author{Nadezda Filimonenkova}

\centerline{Saint-Petersburg State University of Architecture and Civil Engineering}
 \centerline{Vtoraja Krasnoarmejskaja ul., 4}
 \centerline{190005 St.Petersburg, Russia}
\vskip .1in
\centerline{Saint-Petersburg State Polytechnic University}
 \centerline{Polytechnicheskaya ul., 29}
 \centerline{195251 St.Petersburg, Russia}
\vskip .1in
\centerline{nf33@yandex.ru}
\vskip .1in
\thanks{The work is supported by ''Nauchnye Shkoly'' (grant No. 1771.2014.1), Sankt Petersburg State University (grant No. 6.38.670.2013), and the Russian Foundation for Basic Research (grant No. 15-01-07650).}

\begin{abstract}
We study the asymptotic behavior of $C^2$-evolutions $u = u(x,t)$ under a given action of the $m$-Hessian evolution operators and boundary conditions. We obtain sufficient (close to necessary) conditions for the convergence of solutions to the first intial-boundary value problem for the $m$-Hessian evolution equations to stationary functions as $t\rightarrow\infty$. Bibliography: 19 titles.
\end{abstract}

\section{Introduction}
The development of the modern theory of stationary Hessian equations
\cite{Iv83}, \cite{Iv85}, \cite{CNS85}, \cite{F11}
brought out a natural problem of its extension onto evolution Hessian equations. Apparently, the first examples may be found in the book
\cite{Kr.b},
where fully nonlinear equations have been considered in frames of N.V.Krylov theory of Bellman equations. One of them is a parabolic Monge -- Ampere equation, which also has been considered in the paper
\cite{Tso},
\begin{equation}-u_t\det u_{xx}=f>0,\quad(x,t)\in Q_T=\Omega\times(0;T),\quad\Omega\subset R^n.\label{pMA}\end{equation}
Later on some sufficient conditions for existence of admissible solution to the first initial boundary value problem for equations
\begin{equation}-u_t+tr_m^{\frac{1}{m}}u_{xx}=f>0,\quad(x,t)\in Q_T=\Omega\times(0;T),\quad m=1,\dots ,n \label{il}\end{equation}
have been found in the paper
\cite{IL}.
Then equations (\ref{il}) looked as the most natural fully nonlinear analogs of heat equation, $m=1$.

On the other hand the first initial boundary value problem for logarithmic parabolic Hessian equations
\begin{equation}-u_t+\log tr_{m,l}u_{xx}=f,\quad tr_{m,l}u_{xx}:=\frac{tr_m u_{xx}}{tr_l u_{xx}},\quad 0\le l<m\le n\label{log}\end{equation}
had been involved in development concerning Hessian integral norms. These norms were introduced in the paper
\cite{W94},
where admissible solvability of the first initial boundary value problem for (\ref{log}) with $l=0$ had been proved and applied to establish some imbedding theorems. For general choice of $l, m$ this solvability was proved in
\cite{TW98},
where also the asymptotic behavior of admissible solutions was under consideration. In the paper \cite{TW98} the goal was to prove Poincare type inequalities for Hessian integral norms.

Further generalization of equation (\ref{log}) may be found in the paper
\cite{ChW},
where solvability of the relevant parabolic problems were used as a main tool to establish some variational properties of stationary Hessian equations.

In this paper our principal concern is to consider asymptotic behavior of classic solutions of the first initial boundary value problem for $m$-Hessian evolution equations
\begin{equation}E_m[u]=f,\quad u\arrowvert _{\partial^{\prime}Q_T} =\varphi,\quad 1\le m\le n,\label{Dp}\end{equation}
where $\partial^{\prime}Q_T=\left(\Omega\times\{t=0\}\right)\cup \left(\partial \Omega\times[0;T]\right)$, $\Omega$ is a bounded domain in $R^n$,
\begin{equation}E_m[u]:=-u_tT_{m-1}[u]+T_m[u],\quad(x,t)\in \bar Q_T,\label{evo}\end{equation}
$T_m[u]=T_m(u_{xx}):=tr_m u_{xx}$, $u_{xx}$ is the Hesse matrix of $u$ in space variables. We also set by definition $T_0(u_{xx})\equiv1$. Then $T_1[u]=\Delta u$ and (\ref{evo}) is the heat operator, i.e., (\ref{Dp}) is the classic first initial boundary value problem for heat equation.

The $m$-Hessian evolution operators (\ref{evo}) including $m=n+1$ have been introduced in the paper
\cite{I08}.
The idea was to connect parabolic Monge -- Ampere equation (\ref{pMA}) with heat equation and
to find sufficient conditions for solvability of the problem (\ref{Dp}) in the weak (approximate) sense via  parabolic extension of Alecksandrov -- Bakel'man maximum principle
\cite{Kr}, \cite{NU}, \cite{Tso}.
It was assumed in
\cite{I08}
that the first initial boundary value problem for considered equations a priori has $m$-admissible solutions. Later on sufficient conditions for such solvability were presented  in the paper
\cite{IU}.
Notice that in the earlier  paper
\cite{Tso}
this approach was applied to parabolic Monge  --  Ampere equation (\ref{pMA}).

To give a sample of our results consider two dimensional case, i.e., $\Omega\subset R^2$. If $m=1$ operator (\ref{evo}) is the heat operator, while $E_2[u]=-u_t\Delta u+\det u_{xx}$. Assume that $u\in C^{2,1}(\bar\Omega\times[0;\infty))$, $E_2[u]>0$ and let $\bold u$ be a strictly convex in $\bar\Omega$ $C^2$-function. Then there exists $\nu=\nu[\bold u]$ such that $\det\bold u_{xx}\ge\nu>0$.
\begin{Th}
Assume there is a point $x_0\in\Omega$ such that $\Delta u(x_0,0)>0$ and $\lim_{t\rightarrow\infty}|u(x,t)-\bold u(x)|=0$ for $x\in\partial\Omega$, $\lim_{t\rightarrow\infty}|E_2[u]-\det\bold u_{xx}|=0$, $x\in\Omega$. Then $\lim_{t\rightarrow\infty}|u(x,t)-\bold u(x)|=0$ for all $x\in\bar\Omega$.
\end{Th}
It may be said that under conditions of Theorem 1.1 a function $\bold u=\bold u(x)$ attracts evolutions $\{u=u(x,t)\}$ if the functions $\bold f=\det\bold u_{xx}$, $x\in\bar\Omega$ and $\Phi=\bold u\arrowvert_{\partial\Omega}$ attract $\{f=E_2[u]\}$, $(x,t)\in\bar\Omega\times[0;\infty)$, $\{\varphi=u(x,t)\}$, $(x,t)\in\partial\Omega\times[0;\infty)$ respectively. Notice that function $u(x,0)$ from above theorem has no obligations to be convex but it upgrades to convex in time.
\begin{remark}
We notice firstly that inequality $E_2[u]>0$, $(x,t)\in\bar Q_T$ does not admit points $x\in\Omega$ such that $\Delta u(x,0)=0$. Indeed, if so the eigenvalues of $u_{xx}$ are of different sign either vanish. Hence,  $E_2[u](x,0)=\det u_{xx}(x,0)\le0$, what is impossible.

Suppose now that $\Delta u(x,0)<0$, $x\in\bar\Omega$, i.e., $x_0$ from Theorem 1.1 does not exist. If there is a convex solution $\bold v$ to the Dirichlet problem $\det\bold v_{xx}=\det\bold u_{xx}$, $\bold v\arrowvert_{\partial\Omega}=-\Phi$, then $-\bold v$ attracts $u(x,t)$ due to Theorem 1.1.
\end{remark}
\vskip .1in
The second observation from Remark 1.2 carries out
\begin{Cor}
Let $\partial\Omega\in C^{4+\alpha}$, $\alpha>0$ be strictly convex, $u\in C^2(\bar\Omega\times[0;\infty))$. Assume that $E_2[u]>0$ and there is $\bold f\in C^{2+\alpha}(\bar\Omega)$, $\bold f>0$ such that $\lim_{t\rightarrow\infty}|E_2[u]-\bold f|=0$. Assume also $\lim_{t\rightarrow\infty}|u(x,t)|=0$ for $x\in\partial\Omega$. Then a convex solution of the Dirichlet problem
\begin{equation}\det\bold u_{xx}=\bold f,\quad\bold u\arrowvert_{\partial\Omega}=0\label{d2}\end{equation}
attracts $u$ either $-u$.
\end{Cor}
The point of Corollary 1.3 is that the problem (\ref{d2}) has exactly two solutions, which are a convex one $\bold u$ and $-\bold u$. According to Remark 1.2, $u(x,t)$  converges to $\bold u$ if $\Delta u(x,0)>0$ and to $-\bold u$ otherwise.

All these may be extended to an arbitrary dimension but formulation of relevant results requires new geometric and algebraic notions, which are collected in Section 2.

One may see now why parabolic Monge -- Ampere equation (\ref{pMA}), $m=n+1$ got excluded from the set of $m$-Hessian equations (\ref{Dp}). The requirement $T_m[\bold u]>0$ is the basis of our further development and $T_{n+1}[u]\equiv0$ by definition.

In Section 3 we formulate and discuss the main result of the paper. The basis of its proof is the appropriate comparison theorem, what presupposes construction of barriers. The stems of these barriers are chosen as solutions of auxiliary linear first order ordinary differential equations and exposed in Section 4. Section 5 contains a proof of theorem from Section 3.

\section{Notations and definitions}
We denote the space of $N\times N$ symmetric matrices by $Sym(N)$ and by $T_p(S)$ the $p$-traces of $S\in Sym(N)$, which are the sum of all principal $p$-minors of $S$, $1\le p\le N$, $T_0(S):=1$.
\begin{Def}
A matrix $S\in Sym(N)$ is $m$-positive if $S\in K_m$,
\begin{equation}K_m=\{S: T_p(S)>0,\quad p=1,\dots ,m\}.\label{K}\end{equation}
\end{Def}
The cones (\ref{K}), $m=1,\dots ,N$ are the basis of the theory of $m$-Hessian partial differential equations and they admit different definitions. Constructive Definition 1.1 has been introduced in the paper
\cite{Iv83},
while one of the first may be extracted from the paper
\cite{G59}
(see also
\cite{FI}, \cite{IPY}).
Namely,
\begin{Def}
A cone $K_m$ is the component of positiveness of function $T_m(S)$ in $Sym(N)$ containing $S=Id$.
\end{Def}
Due to scalar product $(S^1,S^2):=tr(S^1S^2)$, $Sym(N)$ is a metric space with $||S||^2=(S,S)$. In that sense a cone (\ref{K}) is an open set and $T_m(S)=0$ for $S\in\partial K_m$. Hence, by Definition 2.2 the set of non negative definite $N\times N$ matrices belongs to $\bar K_m$ for all $1\le m\le N$.
\vskip .1in
\begin{remark}
One of the most essential results from
\cite{G59}
allows to substitute $Id$ in Definition 2.2 onto an arbitrary matrix $S_0\in K_m$. Moreover, it follows from G{\aa}rding theory that function $T^{\frac{1}{m}}_m$ is concave in $\bar K_m$
\cite{FI}
and as a consequence, $T_m$ is non negative monotone in $\bar K_m$, i.e., if $S^1, S^2\in\bar K_m$, then $T_m(S^1+S^2)\ge T_m(S^1)$.
\end{remark}

Notice that the set of $N$-positive matrices contains all positive-definite matrices and only them.  Although the simplest way to verify positive definiteness of $S\in Sym(N)$ is to apply well known Sylvester criterion. Quite recently the second author discovered generalized version of Sylvester criterion and extended it to $m$-convex matrices, \cite{F14prep}. 

Denote by $S^{<i_1,\dots, i_k>}\in Sym(N-k)$ a matrix, derived from $S\in Sym(N)$ by crossing out rows and columns numbered by $i_1,\dots, i_k$.

\begin{Th}[Sylvester criterion]\label{S}
Let $S\in Sym(N)$. 

(i) Let $i$ be some fixed index, $1\le i\le N$. Assume that $S^{<i>}$ is $(m-1)$-positive. Then the inequality $T_m(S)>0$ is necessary and sufficient for $m$-positiveness of $S$.

(ii)The existence of at least one collection of numbers $(i_1,\dots, i_{m-1})\subset\{i\}_1^N$ such that
\begin{equation}T_m(S)>0,\quad T_{m-k}(S^{<i_1,\dots, i_k>})>0,\quad k=1,\dots,m-1\label{Sc}\end{equation}
is necessary and sufficient for $m$-positiveness of $S$.
\end{Th}

If in (\ref{Sc}) $m=N$, $(i_1\dots i_{N-1})=(N\dots 2)$, Theorem \ref{S} turns into classic Sylvester criterion.

Our further proceeding will be restricted to the special subspace of $Sym(N)$. Namely, we take into consideration the set
\begin{equation}\bold S^{ev}=\{S^{ev}=(s_{kl})_0^n,\quad s_{00}=s,\quad s_{0i}=s_{i0}=0,\quad S=(s_{ij})^n_1\in Sym(n)\}.\label{evS}\end{equation}
In order to emphasize this speciality we introduce on the subspace (\ref{evS}) new notations for traces $T_p$ and cones (\ref{K}):
\begin{equation}E_m(s,S):=T_m(S^{ev})=sT_{m-1}(S)+T_m(S),\quad 1\le m\le n,\label{Em}\end{equation}
\begin{equation}K_m^{ev}=\{s,S: E_p(s,S)>0,\quad p=1,\dots ,m\}.\label{Ke}\end{equation}

The Sylvester criterion for $m$-positive matrices $1<m<N$ carries out a refined version of description of $K_m^{ev}$. Namely,
\begin{equation}K_m^{ev}=\{s,S: E_m(s,S)>0,\quad S\in K_{m-1}\}.\label{rKe}\end{equation}

Let $\Omega\subset R^n$ be a bounded domain, $Q_T=\Omega\times(0;T)$, $\partial ^{\prime\prime}Q_T=\partial\Omega\times[0;T]$, $\partial^{\prime}Q_T=(\Omega\times\{0\})\cup\partial ^{\prime\prime}Q_T$, $u\in C^{2,1}(\bar Q_T)$. We introduce functional analogs of (\ref{evS}), (\ref{Em}), (\ref{Ke}) letting $S^{ev}[u]$ with $s[u]=-u_t,S[u]=u_{xx}$:
\begin{equation}E_m[u]:=T_m(S^{ev}[u])=-u_tT_{m-1}(u_{xx})+T_m(u_{xx}),\quad 1\le m\le n,\label{uEm}\end{equation}
\begin{equation}\bold K_m^{ev}(\bar Q_T)=\{u\in C^{2,1}(\bar Q_T): S^{ev}[u]\in K_m^{ev}, (x,t)\in (\bar Q_T)\},\label{uKe}\end{equation}
where $u_{xx}$ is Hesse matrix of $u$.

Notice that in contrast to matrix cone (\ref{rKe}) a functional cone (\ref{uKe}) is closed on the bounded in $C^{2,1}(\bar Q_T)$ sets of functions. Indeed, let $u\in C^{2,1}(\bar Q_T)$. Then a matrix set $\{S^{ev}[u], u\in C^{2,1}(\bar Q_T)\}$ is compact in the space $Sym(n+1)$. It means that if $u\in\bold K_m^{ev}(\bar Q_T)$ there exists a value $\nu=\nu[u]>0$ such that $E_m[u]\ge\nu, (x,t)\in\bar Q_T$.

\begin{Def}
We say that operator (\ref{uEm}) is the $m$-Hessian evolutionary operator and a function $u\in \bold K_m^{ev}(\bar Q_T)$ is $m$-admissible in $\bar Q_T$ evolution.
\end{Def}
In our applications we need the following consequence of  Remark 2.3 and (\ref{rKe}).
\begin{lemma}
Let $u\in C^{2,1}(\bar Q_T)$. Assume there is a point $x_0\in\Omega$ such that a matrix $u_{xx}(x_0)$ is $(m-1)$-positive and in addition $E_m[u]>0$, $(x,t)\in\bar Q_T$. Then $u\in\bold K_m^{ev}(\bar Q_T)$.
\end{lemma}

The development of the theory of Hessian equations has brought out some new notions in differential geometry and
the first description of some may be found in
\cite{CNS85}
as necessary conditions for admissible solvability of the Dirichlet problems. In the papers
\cite{FI}, \cite{IPY}, \cite{Iv12} some versions of these requirements were considered independently of differential equations as the set of new geometric notions. Namely, let $\partial \Omega \in R^n$ be $C^2$-hypersurface with position-vector $X=X(\theta)$ and metric tensor $g[\partial\Omega]=(g_{ij})_1^{n-1},$ $g_{ij}=(X_i,X_j)$, $X_i=\partial X/\partial \theta ^i$. In some vicinity of $M_0\in\partial\Omega$ we introduce the set of matrices $\tau=(\tau _i^j)_1^{n-1}$ such that $g^{-1}=\tau^T\tau $ and denote
\begin{equation}X_{(i)}=X_k\tau ^k_i,\quad X_{(ij)}=X_{kl}\tau ^k_i\tau ^l_j,\quad i,j=1,\dots ,n-1. \label{mf}\end{equation}
Notice that $(X_{(i)},X_{(j)})=\delta _{ij}$ and (\ref{mf}) provides Euclidean moving frames for $\partial\Omega$. The freedom  of choice of $\tau$ supplies rotations  in the tangential plane.

The second item in (\ref{mf}) provides the set of symmetric matrices $\mathcal K[\partial\Omega]$,
\begin{equation}\mathcal K[\partial\Omega]=(\mathcal K_{ij})_1^{n-1},\quad \mathcal K_{ij}=(X_{(ij)},\bold n),\label{cm}\end{equation}
where $\bold n$ is the interior to $\partial\Omega$ normal.
\begin{Def}
We say that a matrix (\ref{cm}) is the curvature matrix of $\partial\Omega$ and functions $\bold k_p(M)=T_p(\mathcal K[\partial\Omega])(M)$, $p=1,\dots, n-1$ are the $p$-curvatures of $\partial\Omega$.
\end{Def}

By construction the curvature matrices are geometric invariant in the sense that theirs eigenvalues are the principal curvatures of $\partial\Omega$. On the other hand, $p$-curvatures are absolute geometric invariants admitting natural numbering by $p$ throughout $\partial\Omega.$ It is also remarkable that if $\partial\Omega$ is $C^{2+k}$-smooth, then $\{\bold k_p\}_1^{n-1}$ are $C^k$-smooth.

Definitions 2.1, 2.7 carry out
\begin{Def}
A closed $C^2$-hypersurface $\Gamma $ is $m$-convex at a point $M$ if its curvature matrix is $m$-positive at this point.
\end{Def}
Notice that $m$-positiveness of the curvature matrix does not depend on parametrization.

It follows from $(\ref{K})$ that Definition 2.8 is equivalent to
\begin{Def}
A closed $C^2$-hypersurface $\Gamma $ is $m$-convex at a point $M$ if the first $p$-curvatures of $\Gamma$ are positive up to $m$ at $M$, i.e.,
\begin{equation}\bold k_p[\Gamma](M)>0,\quad p=1,\dots,m. \label{mp}\end{equation}
\end{Def}
And eventually due to Remark 2.3 we have
\begin{Def}
A closed $C^2$-hypersurface $\Gamma $ is $m$-convex if $\bold k_m[\Gamma]>0$.
\end{Def}
As to the principal curvatures of $\Gamma\subset R^{n+1}$, it is known that at least $m$ of them are positive in the points of $m$-convexity but otherwise it is only true for $m=n$, i.e., for strictly convex hypersurfaces in common sense.

\section{Exposition of the problem}

We rewrite the first initial boundary value problem for the $m$-Hessian evolution equation (\ref{Dp}), (\ref{evo}) in terms (\ref{uEm}):
\begin{equation}E_m[u]=-u_tT_{m-1}(u_{xx})+T_m(u_{xx})=f,\quad u\arrowvert_{\partial^{\prime}Q_T}=\varphi, \quad 1\le m\le n, \label{ib}\end{equation}
and relate to (\ref{ib}) the Dirichlet problem
\begin{equation}T_m[\bold u]=\bold f,\quad x\in\Omega,\quad \bold u\arrowvert_{\partial\Omega}=\Phi.\label{mdp}\end{equation}
The main goal of the paper is to present the proof of
\begin{Th}
Assume that
\vskip .1in
$(i)$  $f\ge\nu_m>0$ and there is a point $x_0\in \Omega$ such that a matrix $\varphi_{xx}(x_0,0)$ is $(m-1)$-positive;
\vskip .1in
$(ii)$ there exists a solution $u\in C^{2,1}(\bar\Omega\times[0;\infty))$ to the problem (\ref{ib});
\vskip .1in
$(iii)$ there exists $m$-admissible in $\bar\Omega$ solution $\bold u$ to the problem (\ref{mdp});
\vskip .1in
$(iv)$ $\lim_{t\rightarrow\infty}|f(x,t)-\bold f(x)|=0, x\in\bar\Omega,\quad\lim_{t\rightarrow\infty}|\varphi(x,t)-\Phi(x)|=0, x\in\partial\Omega$;
\vskip .1in

Then $\quad\lim_{t\rightarrow\infty}|u(x,t)-\bold u(x)|=0,\quad x\in\bar\Omega$.
\end{Th}
The following proposition shows that assumption $(ii)$ may reduce the situation to the empty set, whatever smooth data in the problem (\ref{ib}) had been.
\begin{Th}
Let assumption $(i)$ be satisfied. Assume in addition that there is a point $x_1\in\Omega$ such that a matrix $\varphi_{xx}(x_1,0)$ is not $(m-1)$-positive. Then there are no solutions to the problem (\ref{ib}) in $C^{2,1}(\bar Q_T)$, whatever small $T>0$ had been.
\end{Th}
Indeed, assumptions of Theorem 3.2 are incompatible due to Remark 2.3.

The principal goal in the theory of differential equations is to look for close to necessary sufficient  conditions, which guarantee a solvability of the problems under consideration. Concerning the problem (\ref{ib}) such conditions were presented in the paper
\cite{IU}
and we formulate a refined version of Theorem 1.2 from there.
\begin{Th},
Assume that $\partial\Omega\in C^{4+\alpha}$, $\bold k_{m-1}[\partial\Omega]>0$, $f>0$, $f\in C^{2+\alpha,1+\alpha/2}(\bar Q_T)$, $\varphi\in C^{4+\alpha,2+\alpha}(\partial^\prime Q_T)$, $\varphi(x,0)\in\bold K_{m-1}(\bar\Omega)$ and $f, \varphi$ satisfy compatibility condition
\begin{equation} -\varphi_t(x,0)T_{m-1}(\varphi_{xx}(x,0))+T_m(\varphi_{xx}(x,0))-f(x,0)=0,\quad x\in\partial\Omega. \label{cc}\end{equation}
Then there exists a unique in $C^{2,1}(\bar Q_T)$ solution $u$ to the problem (\ref{ib}) and $u\in\bold K_m^{ev}(\bar Q_T)$. Moreover, $u\in C^{2+\alpha,1+\alpha/2}(\bar Q_T)$.
\end{Th}
It follows from Theorem 3.2 that condition $\varphi(x,0)\in\bold K_{m-1}(\bar\Omega)$ is necessary for solvability of the problem (\ref{ib}) in $C^{2,1}(\bar Q_T)$ for odd $n$. In even dimensions it may be substituted on $-\varphi(x,0)\in\bold K_{m-1}(\bar\Omega)$ and in this sense is also necessary.

The analog of Theorem 3.3 for a stationary problem (\ref{mdp}) has been proved in the paper
\cite{CNS85},
Theorem 1.2. We reformulate it in our terminology.
\begin{Th}
Let $\partial\Omega\in C^{4+\alpha}$, $\bold f\in C^{2+\alpha}(\bar\Omega)$, $\Phi(x)\in C^{4+\alpha}(\partial\Omega)$. Assume that $\bold k_{m-1}[\partial\Omega]>0$, $\bold f>0$. Then there exists a unique $m$-admissible solution $\bold u$ to the problem (\ref{mdp}).
\end{Th}
It was discovered in
\cite{CNS85}
that a requirement $\bold k_{m-1}[\partial\Omega]>0$ is necessary for solvability of the problem (\ref{mdp}) if $\bold\Phi=const$ (see also
\cite{FI} ).
More precisely, if there is a point $M_0\in\partial\Omega$ such that $\bold k_{m-1}[\partial\Omega](M_0)<0$, there are no solutions to $m$-Hessian equation (\ref{mdp}) attaining constant Dirichlet condition in $C^2(\bar\Omega)$.
To demonstrate this speciality for $m$-Hessian evolution equations we formulate a non existence result.
\begin{Th}
Let $\partial\Omega\in C^{4+\alpha}$, $f\in C^{2+\alpha,1+\alpha/2}(\bar Q_T)$, $f\ge\nu_m>0$, $\varphi\in  C^{4+\alpha,2+\alpha}(\partial^\prime Q_T)$, $\varphi(x,0)\in\bold K_{m-1}(\bar\Omega)$, $1< m\le n$.

Assume there are $x_0\in\partial\Omega$, $t_0\in (0;T)$, $r>0$ such that $\bold k_{m-1}[\partial\Omega](x_0)<0$, $\varphi_t(x_0,t_0)\ge0$,  $\varphi(x,t_0)=C=const$ in $B_r(x_0)$.

Then there are no $C^{2,1}$-solutions  to the problem (\ref{ib}), (\ref{cc}) for $t\ge t_0$.
\end{Th}
\begin{proof} Indeed, suppose the contrary and there exists a solution $u\in C^{2,1}(\bar Q_{t_0})$. Then it is a unique and with necessity $u\in\bold K_m^{ev}(\bar Q_{t_0})$. Hence, $u(x,t_0)\in\bold K_{m-1}(\bar\Omega)$. But it follows from equation (\ref{ib}) that $T_m(u_{xx})(x_0,t_0)>0$ and by continuity $u(x,t_0)\in\bold K_m(\bar\Omega)\cap B_{r_1}(x_0)$ with some $0<r_1<r$. Moreover, $\partial\Omega\cap B_{r_1}(x_0)$ is a level surface of $m$-admissible function due to assumption $\varphi(x,t_0)=C$ in $B_r(x_0)$ and therefore $\bold k_{m-1}[\partial\Omega](x_0)>0$. This contradiction proves Theorem 3.5.
\end{proof}
\vskip .1in

The crucial step to establish classic solvability of fully nonlinear second order differential equations is to construct a priori estimates of solutions on the basis of various comparison principals. In order to prove Theorem 3.1 we follow this pattern and involve in our reasoning the simplest one.
\begin{Th}
Let functions $v,w\in C(\bar Q_T)$ and $v\in C^{2,1}(Q_T)$, $w\in\bold K_m^{ev}( Q_T)$. Assume that
\begin{equation}E_m[v]-E_m[w]\le0,\quad(x,t)\in Q_T,\quad(w-v)\arrowvert_{\partial^\prime Q_T}\le0.\label{cth}\end{equation}
Then $w\le v$ in $\bar Q_T$.
\end{Th}

\begin{proof} Consider an auxiliary function $w^\varepsilon=w-\varepsilon t$, $\varepsilon>0$ and compute \begin{equation}T_{m-1}[w^\varepsilon]=T_{m-1}[w]>0,\quad E_m[w^\varepsilon]=\varepsilon T_{m-1}[w]+E_m[w]>E_m[w],\quad (x,t)\in Q_T.\label{aE}\end{equation}
Since $w$ is $m$-admissible evolution and due to (\ref{rKe}), (\ref{uEm}), $w^\varepsilon\in\bold K_m^{ev}( Q_T)$. Moreover, it follows from (\ref{cth}), (\ref{aE}) that for $w^\varepsilon$ we have
\begin{equation}E_m[v]-E_m[w^\varepsilon]<0,\quad(x,t)\in Q_T,\quad(w^\varepsilon-v)\arrowvert_{\partial^\prime Q_T}\le0.\label{ast}\end{equation}
Assume that
\begin{equation}\sup_{Q_T}(w^\varepsilon-v)=(w^\varepsilon-v)(x_0,t_0),\quad (x_0,t_0)\in\bar Q_T\backslash\partial^\prime Q_T\label{sp}\end{equation}
and denote $S^{ev}[^.](x_0,t_0)=S^{ev}_0[^.]$. Our choice (\ref{sp}) implies $S_0^{ev}[v-w^\varepsilon]\ge\bold 0$ and since $K^{ev}_m$ is a convex set,  $S^{ev}_0[v]=S^{ev}_0[w^\varepsilon]+S^{ev}_0[v-w^\varepsilon]\in K_m^{ev}$. Then monotonicity of $E_m$ in $K_m^{ev}$ (see Remark 2.3) carries out $E_m[v]-E_m[w^\varepsilon]\ge0$, what contradicts to the first inequality in (\ref{ast}) and assumption (\ref{sp}) is impossible. That means $(x_0,t_0)$ belongs to $\partial^\prime Q_T$ and the second inequality in (\ref{ast}) is valid for all $(x,t)\in\bar Q_T$. This is equivalent to $w-v\le\varepsilon T$ and tending $\varepsilon$ to zero we conclude Theorem 3.6.
\end{proof}
\vskip .1in
There are two possible applications of Theorem 3.6. The first is to construct $v$ satisfying (\ref{cth}) and to estimate $m$-admissible evolution $w$ from above. In this case we say that $v$ is an upper barrier for $w$. Otherwise, $w$ is a lower barrier for $v$.

\section{Auxiliary functions $\theta$, $\sigma$, V}

Denote $g(0)$ by $g_0$ for any function $g$ in Sections 4 and 5. Consider the Cauchy problem for linear ordinary differential equation
\begin{equation}\theta^\prime+b(\theta+h)=0,\quad t\ge0,\quad\theta(0)=\theta_0,\quad b=const>0.\label{od}\end{equation}
Then
\begin{equation}\theta=\exp(-bt)\left(\theta_0-b\int_0^t\exp(b\tau)h(\tau)d\tau\right).\label{th}\end{equation}
The following two propositions are obvious.
\begin{lemma}
Let $\theta$ be a solution to equation (\ref{od}) with  $h=h^+>0$, $(h^+)^\prime(t)\le0$, $h^+_0<\frac{1}{m}$. Assume that $\theta_0+h^+_0\le0$. Then
\begin{equation}\theta(t)+h^+(t)\le0,\quad\theta^\prime(t)\ge0.\label{1th}\end{equation}
If in addition $\lim_{t\rightarrow\infty}h^+(t)=\overline h^+$, then $\lim_{t\rightarrow\infty}\theta(t)=-\overline h^+$,  $\lim_{t\rightarrow\infty}\theta^\prime(t)=0$.
\end{lemma}
\begin{lemma}
Assume that $h=-h^-<0$, $(h^-)^\prime(t)\le0$, $\theta_0-h^-_0\ge0$. Then
\begin{equation}\theta(t)- h^-(t)\ge0,\quad\theta^\prime(t)\le0.\label{2th}\end{equation}
If in addition $\lim_{t\rightarrow\infty}h^-(t)=\underline h^-$, then $\lim_{t\rightarrow\infty}\theta(t)=\underline h^-$,  $\lim_{t\rightarrow\infty}\theta^\prime(t)=0$.
\end{lemma}
Now we introduce the second auxiliary function $\sigma=\sigma(t)$ via the equality
\begin{equation}(1+\sigma)^m=1+m\theta,\quad 1\le m\le n.\label{ths}\end{equation}
Denote $\theta$ from Lemma 4.1 by $\theta^+$ and the relevant solution of equation (\ref{ths}) by $\sigma^+$. Due to assumption $h^+_0<\frac{1}{m}$ function $\sigma^+$ is well defined and it follows from (\ref{ths}), (\ref{1th}) that
\begin{equation}m\theta^+(t)\le\sigma^+(t)<\theta^+(t)<0,\quad(\sigma^+(t))^\prime>0.\label{ts1}\end{equation}

In the situation of Lemma 4.2 we keep to notations $\theta^-$, $\sigma^-$. The analog of relations (\ref{ts1}) in this case sequels (\ref{ths}), (\ref{2th}) and if $\theta_0\ge1$ reads as
\begin{equation}0<\frac{(1+m\theta_0^-)^{\frac{1}{m}}-1}{\theta_0^-}\theta^-(t)<\sigma^-(t)<\theta^-(t),\quad
(\sigma^-(t))^\prime\le0.\label{ts2}\end{equation}

In order to prove Theorem 3.1 we apply Theorem 3.6 with barriers in the form
\begin{equation}V=\sigma(\bold u-A)+\bold u,\quad (x,t)\in \bar Q_T\label{V}\end{equation}
where $\sigma=\sigma(t)$ and a positive constant $A>0$ are to be chosen, $\bold u=\bold u(x)$ is a given $C^2$-function. The following identity is crucial in further reasoning
\begin{equation}E_m[V]=(A-\bold u)T_{m-1}[\bold u]\theta^\prime+mT_m[\bold u](\theta +1).\label{EV}\end{equation}
\vskip .1in
\begin{remark}
We always suppose that $\bold u$ is an $m$-admissible in $\Omega$ solution of the problem (\ref{mdp}) and there are parameters $\nu_m$, $\mu_k$ such that
\begin{equation}0<\nu_m\le\bold f,\quad T_k[\bold u]\le\mu_k,\quad k=m-1,m,\quad x\in\Omega.\label{mn}\end{equation}
\end{remark}

\section{On asymptotic behavior of $m$-Hessian evolutions}

Consider firstly evolutions in a bounded cylinder $Q_T$ and begin with construction of upper bound for $m$-admissible evolutions.
\begin{lemma}
Let $u\in \bold K_m^{ev}(Q_T)\cap C(\bar Q_T)$, $E_m[u]\ge\nu>0$. Assume there exist non increasing functions $h^+_i=h^+_i(t)>0$, $i=1,2$ such that
\begin{equation}(u-\Phi)\arrowvert_{\partial^\prime Q_T}\le h^+_1,\quad\frac{1}{m}\left(1-\frac{E_m[u]}{\bold f}\right)\le h^+_2,\quad(x,t)\in Q_T.\label{ub}\end{equation}
Then
\begin{equation}u(x,t)-\bold u(x)\le m(2mh_1^+(0)+osc_\Omega\bold u)(-\theta^+(t)),\quad(x,t)\in\bar Q_T,\label{uV}\end{equation}
where $\theta^+=\theta$ is given by (\ref{th}) with
\begin{equation}b^+=\frac{m\nu_m}{(2mh_1^+(0)+osc_\Omega\bold u)\mu_{m-1}},\; -h=h^+=\max\left\{\frac{\max\{1-\frac{\nu}{\mu_m};\frac{1}{2}\}}{mh_1^+(0)}h_1^+(t);h_2^+(t)\right\},\label{b1}\end{equation}
$\nu_m$, $\mu_k$, $k=m-1,m$ are the constants in (\ref{mn}) and $\theta^+_0=h_0$.
\end{lemma}

\vskip .1in
\begin{proof} We apply Theorem 3.6 with $w=u$, what reduces our proof to construction of an upper barrier $V=V^+$ (see (\ref{V})) and begin with construction of auxiliary function $\sigma=\sigma^+$ in (\ref{V}). Since $0<h^+_0<\frac{1}{m}$ and due to (\ref{1th}), $\theta^+$ satisfies the inequality $-1<-m\theta^+<0$. Hence function $\sigma=\sigma^+$ is uniquely defined by (\ref{ths}) and relations (\ref{ts1}) are valid. Moreover, due to (\ref{1th}) the inequalities $0<-\sigma{t}\le h^+_1(t)$ are satisfied.

Concerning a constant $A$ in (\ref{V}), we appoint
\begin{equation}A=A^+=A_1^++\sup_\Omega\bold u,\quad A_1^+=2mh^+_1(0).\label{A1}\end{equation}
Then relations (\ref{b1}) provide $u-V^+\le 0$ on the parabolic boundary of $Q_T$. So, the second inequality in (\ref{cth}) got valid.

To confirm the first inequality from (\ref{cth}), we make use of (\ref{EV}) and represent $E_m[V^+]-E_m[u]$ in the form
\begin{equation}E_m[V^+]-E_m[u]=(A-\bold u)T_{m-1}[\bold u]\left((\theta^+)^\prime+\frac{m\bold f}{(A-\bold u)T_{m-1}[\bold u]}\left(\theta^++\frac{\bold f-E_m[u]}{m\bold f}\right)\right).\label{1E}\end{equation}
But due to (\ref{b1}), relations (\ref{1th}), (\ref{od}) and (\ref{1E}) bring out the line
$$
E_m[V^+]-E_m[u]\le(A-\bold u)T_{m-1}[\bold u]\left((\theta^+)^\prime+b(\theta^++h^+)\right)=0,
$$
i.e., the first inequality in (\ref{cth}) is also true.

So, conditions of Theorem 3.6 are satisfied for $w=u$, $v=V^+$. Hence, $u-V^+\le0$ in $\bar Q_T$ and the inequality (\ref{uV}) is a matter of straightforward computing.
\end{proof}
\vskip .1in
It turned out that Theorem 3.6 provides a lower bound for $C^{2,1}$-evolutions under weaker in a way conditions, than (\ref{ub}).
\begin{lemma}
Let $u\in C^{2,1}(Q_T)\cap C(\bar Q_T)$. Assume there exists non increasing function $h^->0$ such that
\begin{equation}(u-\Phi)\arrowvert_{\partial^\prime Q_T}\ge-h^-,\quad\frac{1}{m}\left(1-\frac{E_m[u]}{\bold f}\right)\ge- h^-,\quad(x,t)\in Q_T.\label{lb}\end{equation}
Then
\begin{equation}u(x,t)-\bold u(x)\ge-(A_1^-+osc_\Omega\bold u)\theta^-(t),\quad A_1^-=\frac{h^-_0}{((1+mh^-_0)^{\frac{1}{m}}-1)},\quad(x,t)\in\bar Q_T,\label{Bl}\end{equation}
where $\theta^-=\theta$ is given by (\ref{th}) with
\begin{equation}b^-=\frac{m\nu_m}{(A_1^-+osc_\Omega\bold u)\mu_{m-1}},\quad\theta^-_0=h^-_0, \label{b2}\end{equation}
$\nu_m$, $\mu_k$, $k=m-1,m$ are the constants in (\ref{mn}).
\end{lemma}
\vskip .1in
\begin{proof} This time we are to apply Theorem 3.6 to $w=V^-=\sigma^-(\bold u-(A^-_1+\sup_\Omega\bold u))+\bold u$, $v=u$, where $\sigma^-$ satisfies (\ref{ths}) with $\theta=\theta^-$, what requires $V^-\in\bold K^{ev}_m(Q_T)$. To ensure this inclusion we notice that $V^-_{xx}=(\sigma^-+1)\bold u_{xx}\in K_m\subset K_{m-1}$ for all $x\in\bar\Omega$, $t\in [0;T]$ due to the inequality $\sigma^->0$ and choice of $\bold u$ (see Remark 4.3). Therefore in view of (\ref{rKe}), (\ref{uKe}) it is sufficient to verify the inequality $E_m[V^-]>0$. Indeed, relations(\ref{EV}), (\ref{od}), (\ref{2th}), (\ref{b2}) and the second inequality in (\ref{lb}) provide
\begin{equation}E_m[V^-]>(A^-_1+\sup_\Omega\bold u-\bold u)T_{m-1}[\bold u]((\theta^-)^\prime+ b^-\theta^-)>b^-h^->0,\quad(x,t)\in Q_T.\label{pEV}\end{equation}
Similar to (\ref{pEV}) we infer in $Q_T$ the first inequality in (\ref{cth}), which reads here as $E_m[V^-]-E_m[u]\ge0$.

To estimate $u-V^-$ at the parabolic boundary we make use of (\ref{2th}), (\ref{ts2}) and the choice of $A^-_1$,  (\ref{uV}):
$$
(u-V^-)\arrowvert_{\partial^\prime}Q_T\ge-h^-+A_1^-\frac{(1+mh_0^-)^{\frac{1}{m}}}{h_0^-}\theta^-\ge0.
$$
So, Theorem 3.6 guarantees $u-V^-\ge0$ for all $(x,t)\in\bar Q_T$ and as a consequence (\ref{Bl}).
\end{proof}
\vskip .1in
Notice that if assumptions of above lemmas are satisfied for all $T>0$, then $\bold u$ attracts all evolutions (\ref{ib}), $(i)-(iv)$ in $C(Q)$, $Q=\Omega\times[0;\infty)$, i.e. Theorem 3.1 is indeed true. In fact Lemma 5.1 and Lemma 5.2  give rise to more general proposition.
\begin{Th}
Let $u\in C^{2,1}(Q)\cap C(\bar Q)$. Assume there is a point $(x_0,t_0)\in Q$ such that $u_{xx}(x_0,t_0)$ is an $m$-positive matrix, $0<\nu\le E_m[u]\le\mu$ in $Q_0=\Omega\times[t_0;\infty)$ and assumptions (\ref{ub}), (\ref{lb}) are satisfied in $Q_0$. Then $|u|_Q$ is bounded independently on $t$.

If in addition
\begin{equation}\lim_{t\rightarrow\infty}h(t)=0,\quad h=\max\{h^+;h^-\},\quad t>t_0,\label{0h}\end{equation}
then $\bold u=\bold u(x)$ attracts $u=u(x,t)$ in $C(\bar Q)$.
\end{Th}
We see that Theorem 3.1 is a particular case of Theorem 5.3.

Consider in conclusion the simplest case. Namely, we are to apply Theorem 5.3 to heat operator $E_1[u]=-u_t+\Delta u$. Then a function $\bold u$ has to be a solution of Poisson equation
\begin{equation}\Delta\bold u=\bold f,\quad\bold u\in C^2(\Omega)\cap C(\bar\Omega).\label{eP}\end{equation}
\begin{Cor}
Let $u\in C^{2,1}(Q)\cap C(\bar\Omega\times[0;\infty))$, $|\bold f|\le\mu_1$. Assume that
\begin{equation}\lim_{t\rightarrow\infty}E_1[u](x,t)=\bold f(x),\quad x\in\Omega,\quad\lim_{t\rightarrow\infty}u(x,t)=\bold u(x),\quad x\in\partial\Omega.\label{ho}\end{equation}
Then
\begin{equation}\lim_{t\rightarrow\infty}u(x,t)=\bold u(x),\quad x\in\bar\Omega.\label{he}\end{equation}
\end{Cor}
\vskip .1in
\begin{proof} Since it is not assumed that $\bold f\ge\nu>0$ in (\ref{eP}), equality (\ref{he}) can not be considered as a particular case of Theorem 5.3 for $m=1$. To overcome this complication we represent solution $\bold u$ in the form
\begin{equation}\bold u=\bold u_1+\bold u_2,\quad\bold u_1=\frac{1}{2}(\bold u+Cx^2),\quad\bold u_2=\frac{1}{2}(\bold u-Cx^2),\quad C=\frac{\mu+\nu}{2n} \label{sum}\end{equation}
with some $\nu>0$. It is obvious that $\bold u_1$, $-\bold u_2$ are the solutions of two problems similar to (\ref{eP}) but this time with $\bold f_i\ge\nu/2$, $i=1,2$. Now we associate with (\ref{sum}) evolutions $u_i$, $i=1,2$, $u=u_1+u_2$, satisfying relevant analogs of assumption (\ref{ho}). We see that all conditions of Theorem 5.3, $m=1$ hold for $u_1$, $-u_2$, what carries out desirable relation (\ref{he}).
\end{proof}


\begin{thebibliography}{12}

\bibitem{CNS85} Caffarelli L., Nirenberg L., Spruck J., The Dirichlet problem for nonlinear second order elliptic  equations III. Functions of the Hessian. Acta Math. 155, 1985, 261-301.
    
\bibitem{ChW}
Chou Kai-Seng, Wang Xu-Jia, A variational theory of the Hessian equations. Comm. Pure Appl. Math., 54 (2001), pp. 1029-1064.

\bibitem{G59} Garding L., An inequality for hyperbolic polynomials. J. Math. Mech. 8 ,1959, 957-965.

\bibitem{Iv83} Ivochkina N. M., A description of the stability cones generated by differential operators of Monge
- Ampere type. Mat. Sb. 122 (164) (1983), 265--275; English transl. in Math. USSR Sb. 50 (1985).

\bibitem{Iv85} Ivochkina N. M., Solution of the Dirichlet problem for some equations of Monge - Ampere type. Mat. Sb. 128 (170), 1985, 403-415; English transl. in Math. USSR Sb. 56 (1987)
    
\bibitem{Iv12} Ivochkina N. M., From G{\aa}rding cones to p-convex hypersurfaces. J. of Math. Sci. 201 (2014), no. 5, 634--644.

\bibitem{IU} Ivochkina N.M., On classic solvability of the $m$-Hessian evolution equation. AMS Transl. 229 (2010), Series 2, 119-129.

\bibitem{I08} Ivochkina N.M., On approximate solutions to the first initial boundary value problem for the m-Hessian evolution equations. J. Fixed Point Th. Appl. 4 (2008), no.1, 47-56.
    
\bibitem{IL}Ivochkina N.M.,Ladyzhenskaya O.A., On parabolic problems generated by some symmetric functions of the eigenvalues of Hessian. Top. Meth. Nonlinear Anal. 4, 1994, 19-29.

\bibitem{IPY} Ivochkina N. M., Prokof'eva S. I., Yakunina G. V., The G{\aa}rding cones in the modern theory of fully nonlinear second order differential equations. J. of Math. Sci., 184, 2012, no.3, 295-315.

\bibitem{F11}
Filimonenkova N.V., On the classical solvability of the Dirichlet problem for nondegenerate m-Hessian equations. J. of Math. Sci., 178 (2011), pp. 666--694.
    
\bibitem{F14prep}
Filimonenkova N. V., Sylvester's criterion for m-positive matrices.  SPb. Math. Society Preprint 2014-7

\bibitem{FI} Filimonenkova N. V., Ivochkina N. M., On the backgrounds of the theory of m-Hessian equations. Comm. Pure Appl. Anal., 12 (2013), no.4, 1687--1703.
    
\bibitem{Kr} Krylov N. V., Sequences of convex functions and estimates of the maximum of the solutions of a parabolic equation.
    Sibirsk. Mat. Zh. 17 (1976), 226-236 (Russian); English transl.: Siberian Math. J. 17 (1976), 226-236.

\bibitem{Kr.b} Krylov N. V., Nonlinear elliptic and parabolic equations of second order. Nauka, Moscow (1985). English transl.: Reidel, Dordrecht, 1987.

\bibitem{NU} Nazarov A.I., Uraltseva N.N., Convex-monotone hulls and an estimate of the maximum of the solution of a parabolic
    equation. Zap. Nauchn. Sem. LOMI 147 (1985), 95-109 (Russian); English transl.: J. Soviet Math. 37 (1987), 851-859.
    
\bibitem{TW98} Trudinger N.S., Wang X.-J., A Poincare type inequality for Hessian integrals. Calc. Var. and PDF, 6 (1998), 315-328.

\bibitem{Tso}Tso K., On an Alecksandrov--Bakel'man type maximum principle for second-order parabolic equations. Comm. Partial Differ. Equations 10 (1985), 543-553.

\bibitem{W94}
Wang X.-J., A class of fully nonlinear elliptic equations and related functionals. Indiana Univ. Math. J., 43 (1994), pp. 25--54.

\end{thebibliography}
\end{document}